\documentclass[11pt]{article}
\usepackage{graphicx}
\usepackage{amssymb}
\usepackage{epstopdf}
\usepackage{amsmath,amsthm,amscd,amssymb}
\usepackage{latexsym}
\usepackage[colorlinks,citecolor=red,pagebackref,hypertexnames=false]{hyperref}
\usepackage{geometry}                
\numberwithin{equation}{section}

\theoremstyle{plain}
\newtheorem{theorem}{Theorem}[section]
\newtheorem{lemma}[theorem]{Lemma}
\newtheorem{corollary}[theorem]{Corollary}

\theoremstyle{definition}
\newtheorem{definition}[theorem]{Definition}

\newtheorem{case[theorem]}{Case}

\theoremstyle{remark}
\newtheorem{remark}[theorem]{Remark}

\numberwithin{equation}{section}

\begin{document}

\title{Long paths in the distance graph over large subsets of vector spaces over finite fields}

\author{M. Bennett\footnote{University of Rochester}, J. Chapman\footnote{Lyons College}, D. Covert\footnote{University of Missouri - Saint Louis (Corresponding Author)}, D. Hart\footnote{Rockhurst University, NSF Grant \#1242660}, A. Iosevich\footnote{University of Rochester, NSF Grant DMS10-45404 }, and J. Pakianathan\footnotemark[1]}

\maketitle

\begin{abstract} Let $E \subset {\Bbb F}_q^d$, the $d$-dimensional vector space over a finite field with $q$ elements. Construct a graph, called the distance graph of $E$, by letting the vertices be the elements of $E$ and connect a pair of vertices corresponding to vectors $x,y \in E$ by an edge if $||x-y||={(x_1-y_1)}^2+\dots+{(x_d-y_d)}^2=1$. We shall prove that if the size of $E$ is sufficiently large, then the distance graph of $E$ contains long non-overlapping paths and vertices of high degree. \end{abstract} 

\section{Introduction}

\vskip.125in

The classical Euclidean distance graph can be described as follows. Let the vertices of the graph be the points of the Euclidean plane. Connect two vertices by an edge if the Euclidean distance between the corresponding vectors is equal to $1$. A very interesting open question is the exact value of the chromatic number of this graph, the minimal number of colors required so that no two points of the same color are a distance $1$ apart. It is known that the answer is at least four and at most seven. 

In this paper, we continue the investigation of the corresponding distance graph in ${\Bbb F}_q^d$, the $d$-dimensional vector space over the finite field with $q$ elements. Once again, the vertices of the graph are the points of ${\Bbb F}_q^d$ and two vertices $x,y$ are connected by an edge if $||x-y||=1$, where 
$$ ||x||=x_1^2+x_2^2+\dots+x_d^2.$$ 

For some previous results on the properties of the distance graph in ${\Bbb F}_q^d$ see, for example, \cite{MMST}, \cite{V11} , \cite{CHISU} and \cite{CEHIK10}. In this paper we consider a much more complicated case when instead of taking all points in ${\Bbb F}_q^d$ as the vertices of the distance graph, we merely consider points in a subset of ${\Bbb F}_q^d$ of a sufficiently large size. To see that this formulation is meaningful, recall that the 4th listed author and Misha Rudnev proved in \cite{IR07} that if $E \subset {\Bbb F}_q^d$, $d \ge 2$, and $t \neq 0$, then 
$$ |\{(x,y) \in E \times E: ||x-y||=t \}|=\frac{{|E|}^2}{q}+R(t),$$ where 
$$ |R(t)| \leq 2q^{\frac{d-1}{2}}|E|.$$ 

Here and throughout, $|S|$ denotes the number of elements in a (finite) set $S$. In particular, if $t=1$ and $|E|>4q^{\frac{d+1}{2}}$, then 
$$ |\{(x,y) \in E \times E: ||x-y||=1 \}| \ge \frac{{|E|}^2}{2q}.$$ or, in other words, the number of edges in the distance graph of $E$ is at least 
$\frac{{|E|}^2}{2q}$. 

\subsection{Main results} We know that there are many edges and an interesting question is whether the distance graph of $E \subset {\Bbb F}_q^d$ must contain a long path. It turns out that the answer is affirmative. More precisely, we have the following result. 

\vskip.125in 

\begin{theorem} \label{mainfinitefield} Let $E \subset {\mathbb F}_q^d$, where $d \ge 2$ and $|E| > \frac{2k}{\ln 2} q^{\frac{d+1}{2}}$. Suppose that $t_i \not=0$, $1 \leq i \leq k$, and let $\vec{t}=(t_1, \dots, t_k)$. Define 
\[ C_k(\vec{t})=|\{(x^1, \dots, x^{k+1}) \in E \times \dots \times E: ||x^i-x^{i+1}||=
t_i, \ 1 \leq i \leq k \}|.\]
Then  \[ C_k(\vec{t})=\frac{{|E|}^{k+1}}{q^k}+{\mathcal D}_k(\vec{t}), \] where 
\[ |{\mathcal D}_k(\vec{t})| \leq \frac{2k}{\ln 2} q^{\frac{d+1}{2}} \frac{{|E|}^k}{q^k}.\] 
In particular, since $|E|> \frac{2k}{\ln 2} q^{\frac{d+1}{2}}$, $C_k(\vec{t})>0$.
\end{theorem} 

\vskip.125in 

\begin{definition} A path of length $k$ in a simple graph $G$ is a sequence of vertices $v^1,v^2, \dots, v^{k+1} \in G$ such that each pair $v^i, v^{i+1}$, $1 \leq i \leq d$ is connected by an edge. We say that a path of length $k$ is non-overlapping if all the $v^j$s in the definition are distinct. \end{definition} 

\vskip.125in

\begin{corollary} \label{longpath} With the notation above, suppose that $|E| \geq \frac{4k}{\ln 2}q^{\frac{d+1}{2}}$. Then the distance graph of $E$ contains a non-overlapping chain of length $k$ of every type.
\end{corollary} 

\vskip.125in 

\begin{remark} In qualitative terms, Corollary \ref{longpath} says that if $|E|$ is, say, $\approx q^{d-\epsilon}$ for some $\epsilon>0$, then the distance graph of $E$ contains a path of length $\approx q^{\frac{d-1}{2}-\epsilon}$. We do not know to what extent this conclusion is sharp. 
\end{remark} 

In addition to studying chains in a distance graph, we study the following closely related configuration.  Given a set $E \subset \mathbb{F}_q^d$, fix a point $x \in E$, and count all of the vectors that are simultaneously some prescribed distance from $x$ (we call such configurations $k$-stars).  We show that if $E$ is of sufficiently large cardinality, then we achieve the statistically correct number of $k$-stars.  More precisely, we prove the following.

\begin{theorem} \label{kstars}
Let $E \subset \mathbb{F}_q^d$, and suppose $t_i \neq 0, 1 \leq i \leq  k$.  For $\vec{t} = (t_1, \dots, t_k)$, define
\[
\nu_k(\vec t) = \left|\left\{(x, x^1, \dots, x^k) \in E^{k+1} : \| x - x^i \| = t_i, x^i = x^j \iff i = j, \quad  1 \leq i \leq k\right\}\right|.
\]

If $|E| > 12q^\frac{d+1}{2}$, then
$\displaystyle \nu_k(\vec t) > 0$ for any $\displaystyle k < \frac{|E|}{12q^\frac{d+1}{2}}$.  

\vskip.125in 

If $|E| > 12q^\frac{d+3}{2}$, then
$\nu_k(\vec t) > 0$ for any $\displaystyle k < \frac{|E|}{12q}$.

\end{theorem}

\vskip.25in 

\subsection{Fourier Analysis in $\mathbb{F}_q^d$}
Let $f : \mathbb{F}_q^d \to \mathbb{C}$ and $\chi$ a nontrivial additive character of $\mathbb{F}_q$.  Then,
\[ \widehat{f}(m) = q^{-d} \sum_{x \in \mathbb{F}_q^d} \chi(- m \cdot x) f(x). \]
We have the following Plancherel and inversion formulas:
\[ \sum_{m \in \mathbb{F}_q^d} |\widehat{f}(m)|^2 = q^{-d} \sum_{x \in \mathbb{F}_q^d} |f(x)|^2 \]
\[ f(x) = \sum_{m \in \mathbb{F}_q^d} \widehat{f}(m) \chi(m \cdot x). \]

The proofs are straightforward. See, for example, \cite{IR07} and the references contained therein. 

\vskip.125in 

\section{Proof of Theorem \ref{mainfinitefield}} 

\vskip.125in 

We shall need the following functional version of the arithmetic analog of Falconer's result proved in \cite{CHIKR10}. 
\begin{theorem} \label{functionalfffalconer} Let $f,g: {\mathbb F}_q^d \to {\mathbb R}^{+}$. Let $S_t=\{x \in {\mathbb F}_q^d: ||x||=t \}$, where $||x||=x_1^2+\dots+x_d^2$ and $t \not=0$. Then 
\[ \sum_{x,y \in {\mathbb F}_q^d} f(x)g(y)S_t(x-y)=\frac{|S_t|}{q^d} \cdot {||f||}_1 \cdot {||g||}_1+D(f,g),\] where 
\begin{equation} \label{remainderest} |D(f,g)| \leq 2q^{\frac{d-1}{2}} {||f||}_2 {||g||}_2. \end{equation} 
\end{theorem} 

To prove Theorem \ref{functionalfffalconer}, we write 
\[ \sum_{x,y \in {\mathbb F}_q^d} f(x)g(y)S_t(x-y)\]
\begin{equation} \label{falmost}=q^{2d} \sum_{m \not=(0, \dots, 0)} \widehat{f}(m) \overline{\widehat{g}}(m) \widehat{S}_t(m)
+ \frac{|S_t|}{q^d} \cdot {||f||}_1 \cdot {||g||}_1.\end{equation}

We need the following basic facts about the discrete sphere. See, for example, \cite{IR07}. 
\begin{lemma} \label{sphere}
Let $S_t = \{x \in \mathbb{F}_q^d : \| x \| = t\}$.  Identify $S_t$ with its characteristic function.  For $t \neq 0$,
\begin{equation} \label{sizeofsphere}
|S_t| = q^{d-1}(1 + o(1)).
\end{equation}
If $t \neq 0$ and $m \neq (0, \dots , 0)$, then
\begin{equation} \label{decaysphere}
| \widehat{S}_t(m) | \leq 2q^{- \frac{d+1}{2}}.
\end{equation}
\end{lemma}

Plugging \eqref{decaysphere} into \eqref{falmost} and applying Cauchy-Schwarz, we obtain \eqref{remainderest}. This completes the proof of Theorem \ref{functionalfffalconer}. 

We shall prove Theorem \ref{mainfinitefield} in the case $t=(1, \dots, 1)$ for the sake of ease of exposition, but the reader can easily convince oneself that the argument extends to the general case. Let $C_n = C_n(1,\dots , 1)$.  The basic mechanism of our proof is encapsulated in the following claim. 
\begin{lemma} \label{structurefinitefield} With the notation above, we have 
\[ C_{2k+1}=q^{-1}C_k^2+R_{2k+1}, \]
\[ C_{2k}=q^{-1}C_kC_{k-1}+R_{2k}, \]

where 
\[  |R_{2k+1}| \leq 2q^{\frac{d-1}{2}} C_{2k}\]
and
\[  |R_{2k}| \leq 2q^{\frac{d-1}{2}} \sqrt{C_{2k} C_{2k-2}}.\]

\end{lemma} 

To prove the lemma, define 
\[ f_1(x)=(E*S)(x)E(x),\] where $S=S_1$, and let 
\[ f_{k+1}(x)=(f_k*S)(x)E(x).\] 

Unraveling the definition of $C_{2k+1}$, we see that it equals 
\[ \sum_{x,y} f_k(x)f_k(y)S(x-y),\] which, by Theorem \ref{functionalfffalconer} is equal to 
\[ q^{-1} {\left(\sum_x f_k(x)\right)}^2+R_{2k+1},\] where 
\[ |R_{2k+1}| \leq 2q^{\frac{d-1}{2}} {||f_k||}_2^2.\]
Similarly, 
\begin{align*}
C_{2k} &= \sum_{x,y} f_k(x) f_{k-1}(y) S(x-y)\\
&= q^{-1} \sum_x f_k(x) \cdot \sum_y f_{k-1}(y)+R_{2k}, 
\end{align*}
where 
\[ |R_{2k}| \leq 2q^{\frac{d-1}{2}} {||f_k||}_2 {||f_{k-1}||}_2.\] 

By a direct calculation, 
\[ {||f_k||}_1=C_k\] and 
\[ {||f_k||}_2^2=C_{2k}.\] 
Putting everything together we recover the conclusions of Lemma \ref{structurefinitefield}.

\begin{lemma}

$$C_{n} \leq |E| \left(\frac{|E|+2q^\frac{d+1}{2}}{q}\right)^n.$$

\end{lemma}

\begin{proof}

Let $\displaystyle X = \frac{|E|+2q^\frac{d+1}{2}}{q}$. We know that $C_1 \leq |E|X$. Now we induct on the chain length. Suppose it holds for $C_k$ when $k<n$.
From the previous lemma,

\begin{align*} C_{2k+1} &\leq \frac{C_k^2}{q} + 2q^\frac{d-1}{2}C_{2k}
\\
&\leq \frac{(|E|X^k)^2}{q} + 2q^\frac{d-1}{2} |E|X^{2k} = |E|X^{2k+1}.
\end{align*}

Completing the square in the expression for $C_{2k}$ in the previous lemma gives us 

\begin{align*}
C_{2k} &\leq \frac{C_kC_{k-1}}{q} + 2q^{d-1}C_{2k-2} + 2\sqrt{q^{2d-2}C_{2k-2}^2 + q^{d-2}C_{2k-2}C_kC_{k-1}}
\\
&\leq \frac{|E|^2X^{2k-1}}{q} + 2q^{d-1}|E|X^{2k-2} + 2\sqrt{q^{2d-2}|E|^2X^{4k-4} + q^{d-2}|E|^3X^{4k-3}}
\\
&= \frac{|E|^2X^{2k-1}}{q} + 2q^{d-1}|E|X^{2k-2}\left(1+\sqrt{1 + \frac{|E|X}{q^d}}\right)
\\
&= \frac{|E|^2X^{2k-1}}{q} + 2q^{d-1}|E|X^{2k-2}\left(1+\sqrt{\frac{q^{d+1} + |E|^2 + 2|E|q^\frac{d+1}{2}}{q^{d+1}}}\right)
\\
&= \frac{|E|^2X^{2k-1}}{q} + 2q^{d-1}|E|X^{2k-2}\left(2+\frac{|E|}{q^\frac{d+1}{2}}\right) = |E|X^{2k}.
\end{align*}

\end{proof}

We are now ready to complete the proof of Theorem \ref{mainfinitefield}.

\begin{proof}

By the previous lemma, we have

\begin{align*}C_{n} &\leq |E| \left(\frac{|E|+2q^\frac{d+1}{2}}{q}\right)^n
\\
&= \frac{|E|}{q^n}\sum_{i=0}^n {n \choose i}|E|^{n-i} \left(2q^\frac{d+1}{2} \right)^i
\\
&= \frac{|E|^{n+1}}{q^n} + \frac{2q^\frac{d+1}{2}|E|}{q^n} \sum_{i=1}^{n} {n \choose i}|E|^{n-i} \left(2q^\frac{d+1}{2}\right)^{i-1}
\end{align*}

By assumption, $\displaystyle q^\frac{d+1}{2} \leq \frac{|E|\ln 2}{2n}$, so
\begin{align*}
C_{n} &\leq \frac{|E|^{n+1}}{q^n} + \frac{2q^\frac{d+1}{2}|E|^n}{q^n} \sum_{i=1}^{n} {n \choose i}\left(\frac{\ln2}{n} \right) ^{i-1}
\\
&= \frac{|E|^{n+1}}{q^n} + \frac{2n}{\ln 2}q^\frac{d+1}{2}\frac{|E|^n}{q^n} \left(\left(1 + \frac{\ln 2}{n} \right)^n -1 \right)
\\
&\leq \frac{|E|^{n+1}}{q^n} + \frac{2n}{\ln 2}q^\frac{d+1}{2}\frac{|E|^n}{q^n} (e^{\ln 2} -1) =  \frac{|E|^{n+1}}{q^n} + \frac{2n}{\ln 2}q^\frac{d+1}{2}\frac{|E|^n}{q^n}.
\end{align*}

For the lower bound, we use induction. We know that

$$C_1 \geq \frac{|E|^2}{q} - \frac{2}{\ln 2}q^{\frac{d-1}{2}}|E|.$$

Suppose that for $k < 2n+1$, we have

$$C_k \geq \frac{|E|^{k+1}}{q^k} - \frac{2k}{\ln 2}q^{\frac{d+1}{2}}\frac{|E|^k}{q^k}.$$

We have already shown that 

$$\left|C_{2n+1} - \frac{C_n^2}{q} \right| \leq 2q^\frac{d-1}{2}C_{2n}.$$

This implies

\begin{align*}
C_{2n+1} &\geq\frac{\left(\frac{|E|^{n+1}}{q^n} - \frac{2n}{\ln 2}q^{\frac{d+1}{2}}\frac{|E|^n}{q^n}\right)^2}{q} - 2q^\frac{d-1}{2}\left(\frac{|E|^{2n+1}}{q^{2n}} + \frac{2n}{\ln 2}q^{\frac{d+1}{2}}\frac{|E|^{2n}}{q^{2n}}\right)
\\
&= \frac{|E|^{2n+2}}{q^{2n+1}} -q^\frac{d+1}{2}\frac{|E|^{2n+1}}{q^{2n+1}} \left(\frac{4n}{\ln 2} - q^\frac{d+1}{2}\frac{4n^2}{|E|(\ln 2)^2} + 2 + \frac{4n}{|E|\ln 2}q^{\frac{d+1}{2}} \right).
\end{align*}

Observing that

$$\frac{4n}{|E|\ln 2}q^{\frac{d+1}{2}} - q^\frac{d+1}{2}\frac{4n^2}{|E|(\ln 2)^2} < 0,$$

we have 

$$C_{2n+1} \geq \frac{|E|^{2n+2}}{q^{2n+1}} - \frac{4n+2}{\ln 2}q^\frac{d+1}{2}\frac{|E|^{2n+1}}{q^{2n+1}}.$$

A nearly identical argument gives us the analog for $C_{2n}$.

\end{proof}

\section{Proof of Corollary \ref{longpath}}

In analogy with $f_k$, let $g_k(x)$ be the number of non-overlapping $k$-paths in $E$ beginning at $x$. Then the total number of $k$-paths is

$$G_k = \sum_x g_k(x)$$

Given a non-overlapping $k$-path $(v_0, v_1, v_2, \dots, v_n)$, we must be able to find at least $-n + \sum_y E(y)S(x-y)$ choices of $x$ so that the path $(x, v_0, v_1, \dots, v_n)$ is also non-overlapping. There may be some values of $i$ for which $S(v_i - v_0) = 1$, which is why we must subtract $n$. Otherwise we may be including some overlapping paths in the count.

We then have the following recurrence relation:

$$G_{n+1} \geq \sum_x g_n(x)\left(-n + \sum_y E(y)S(x-y) \right) = -nG_n + \sum_{x,y} g_n(x)E(y)S(x-y)$$

We estimate the sum just as we did in the proof of lemma \ref{structurefinitefield} to get

$$\sum_{x,y} g_n(x)E(y)S(x-y) \geq \frac{|G_n||E|}{q} - 2q^\frac{d-1}{2}|E|^{1/2} \left( \sum_x (g_n (x))^2 \right)^{1/2}$$

We also know that 

$$\sum_x (g_n (x))^2 \leq \sum_x (f_n (x))^2 = C_{2n}.$$

By theorem \ref{mainfinitefield}, we have

$$C_{2n} \leq \frac{|E|^{2n+1}}{q^{2n}} + \frac{4n}{\ln 2}q^\frac{d+1}{2}\frac{|E|^{2n}}{q^{2n}},$$

and by assumption, $|E| \geq \frac{4n}{\ln 2}q^\frac{d+1}{2}$. Thus 

$$C_{2n} \leq 2\frac{|E|^{2n+1}}{q^{2n}}.$$

Moreover

$$G_n \leq C_n \leq \frac{|E|^{n+1}}{q^n} + \frac{2n}{\ln 2} q^\frac{d+1}{2}\frac{|E|^n}{q^n}$$

We will induct on the chain length and assume that

$$G_k \geq \frac{|E|^{k+1}}{q^k} - \frac{4k}{\ln 2}q^\frac{d+1}{2}\frac{|E|^{k}}{q^{k}}.$$

Putting everything together, we get
\begin{align*}
G_{k+1} &\geq \frac{|G_k||E|}{q} - kG_k - 2q^\frac{d-1}{2}\sqrt{2\frac{|E|^{2k+2}}{q^{2k}}}
\\
&\geq \frac{|E|^{k+2}}{q^{k+1}} - \frac{4k}{\ln 2}q^\frac{d+1}{2}\frac{|E|^{k+1}}{q^{k+1}} -
			k\frac{|E|^{k+1}}{q^k} - \frac{4k^2}{\ln 2}q^\frac{d+1}{2}\frac{|E|^{k}}{q^{k}} - 2\sqrt{2}q^\frac{d+1}{2} \frac{|E|^{k+1}}{q^{k+1}}
\\
&= \frac{|E|^{k+2}}{q^{k+1}} - q^\frac{d+1}{2}\frac{|E|^{k+1}}{q^{k+1}}\left(\frac{4k}{\ln 2} + \frac{k}{q^\frac{d-1}{2}} + \frac{4k^2q}{|E|\ln 2} + 2\sqrt{2}\right).
\end{align*}

By assumption,  $|E| \geq \frac{4k}{\ln 2}q^\frac{d+1}{2}$, so

$$G_{k+1} \geq \frac{|E|^{k+2}}{q^{k+1}} - q^\frac{d+1}{2}\frac{|E|^{k+1}}{q^{k+1}}\left(\frac{4k}{\ln 2} + \frac{2k}{q^\frac{d-1}{2}} + 2\sqrt{2}\right).$$

Lastly, since $|E| \leq q^d$, we have $k \leq \frac{\ln 2}{4} q^\frac{d-1}{2}$, giving us

$$G_{k+1} \geq \frac{|E|^{k+2}}{q^{k+1}} - q^\frac{d+1}{2}\frac{|E|^{k+1}}{q^{k+1}}\left(\frac{4k}{\ln 2} +\frac{\ln 2}{2} + 2\sqrt{2}\right) \geq \frac{|E|^{k+2}}{q^{k+1}} - \frac{4k+4}{\ln 2}q^\frac{d+1}{2}\frac{|E|^{k+1}}{q^{k+1}}.$$

\section{Proof of Theorem \ref{kstars}}

Suppose $\vec{t}$ is $k$-dimensional with all nonzero indices. Notice that a $k$-star with edge-lengths given by $\vec{t}$ can also be described by any permutation of the indices of $\vec{t}$. Hence we assume without loss of generality that

$$t_1 = t_2 = \dots = t_{i_1} = 1 ; t_{i_1+1} = \dots = t_{i_2} = 2 ; \dots ; t_{i_{q-2}+1} = \dots = t_{i_{q-1}} = q-1$$

with $i_{q-1} = k$.  Let $h(x) = \#\{ y \in E : \|x-y\| = 1 \}$. (The choice of 1 as the length is arbitrary.) We may equivalently write

$$ h(x) = \sum_y E(y)S(x-y).$$

We begin by estimating $H_n =  \#\{x \in E : h(x) \geq n\}$. That is, $H_n$ is the number of points of $E$ from which an $n$-star with all edge lengths 1 can be made.

We use Cauchy-Schwarz to get

$$\left( \sum_{x, h(x) \geq n} E(x)h(x) \right)^2 \leq \left( \sum_{x, h(x) \geq n} E(x)(h(x))^2 \right) \left( \sum_{x, h(x) \geq n} E(x) \right) \leq H_n\sum_{x} E(x)(h(x))^2.$$

Notice that the sum on the right-hand side is

$$\sum_{x,y,z} E(x)E(y)E(z)S(x-y)S(x-z).$$

By theorem \ref{mainfinitefield}, this is less than $\displaystyle \frac{|E|^3+6q^\frac{d+1}{2}|E|^2}{q^2}$.

For the left-hand side, we notice first that 

$$\sum_{x, h(x) \geq n} E(x)h(x) \geq \sum_{x} E(x)(h(x)-n) = -n|E| + \sum_{x} E(x)h(x).$$

Using theorem \ref{mainfinitefield} again, we see that the left side is bounded below by 

$$\left( \frac{|E|^2-2q^\frac{d+1}{2}|E|}{q} -n|E|\right)^2.$$

Solving for $H_n$ gives

$$H_n \geq |E| - 10q^\frac{d+1}{2} - 2qn.$$

We now return to our vector $\vec{t}$. We now know that there are at least $H_{i_j}$ points in $E$ with $i_j$-stars having all edge lengths $j$. By pigeonholing, there must be at least

$\displaystyle |E| - \sum_{j=1}^{q-1} \left(|E|-H_{i_j} \right)$ points in $E$ from which emanate $k$-stars given by $\vec{t}$.

To get something larger than zero, we just need

$$|E| > \sum_{j=1}^{q-1} \left(|E|-H_{i_j} \right).$$

Note that there are at most $k$ values of $j$ for which $H_{i_j} \neq |E|$. 

Since

$$ \sum_{j=1}^{q-1} |E|-H_{i_j} < \sum_{j=1}^{\min \{k,q-1\}}  \left(10q^\frac{d+1}{2} + 2q(i_j-i_{j-1}) \right),$$

it suffices to have $|E| > \min\{k,q\}10q^\frac{d+1}{2} + 2qk$.  We may also put it thus:

If $|E| > 12q^\frac{d+1}{2}$, then $E$ contains an $\displaystyle \frac{|E|}{12q^\frac{d+1}{2}}$-star of every type.

If $|E| > 12q^\frac{d+3}{2}$, then $E$ contains an $\displaystyle \frac{|E|}{12q}$-star of every type.

\end{document}